\documentclass{amsart}
\usepackage{amssymb, latexsym}
\usepackage{amscd}

\newtheorem{theorem}{Theorem}
\newtheorem{lemma}{Lemma}
\newtheorem{remark}{Remark}
\newtheorem{definition}{Definition}

\begin{document}

\title{On Tame Subgroups of Finitely Presented Groups}
\author{Rita Gitik}
\email{ritagtk@umich.edu}
\address{Department of Mathematics \\ University of Michigan \\ Ann Arbor, MI, 48109}  

\date{\today}

\begin{abstract} 
We prove that the free product of two finitely presented locally tame groups
is locally tame and describe many examples of tame subgroups of finitely presented groups.
We also include some open problems related to tame subgroups.
\end{abstract}

\maketitle

\section{Introduction}

A $3$-manifold is called topologically tame if it is homeomorphic to the 
interior of a compact $3$-manifold.

Marden's tameness conjecture  \cite{Ma} states that any hyperbolic $3$-manifold
with finitely generated fundamental group is topologically tame.
The tameness conjecture became one of the central questions
in the theory of hyperbolic $3$-manifolds. The conjecture has been established
by Agol in \cite{Ag} and, independently,  by Calegari and Gabai in \cite{C-G}.
Alternative proofs of the conjecture were given by Soma in \cite{So} and Bowditch 
in \cite{Bo}.

The tameness conjecture is closely related to Simon's missing boundary 
manifold conjecture.

A $3$-manifold $M$ is called a missing boundary manifold  if it can be 
embedded in a compact manifold $\bar M$  such that $\bar M - M$
is a closed subset of the boundary of $\bar M$. 

Simon conjectured in \cite{Sim} that if $M_0$ is a compact
orientable irreducible $3$-manifold, and $M$ is the cover of $M_0$
corresponding to  a finitely generated subgroup of $\pi_1(M_0)$, then
$M$ is a missing boundary manifold. 

Long and Reid proved that  Marden's conjecture implies Simon's 
conjecture for $3$-manifolds which admit a geometric decomposition.
Their proof appeared in \cite{Ca}.

The resolution of Thurston's Geometrization Conjecture by Perelman in 2003 showed that
Simon's conjecture holds for all compact orientable irreducible $3$-manifolds.
Accounts of Perelman's work were given by many mathematicians, see for example, 
\cite{B-B-B-M-P} and \cite{K-L1}. 

Tucker proved in \cite{Tuc}  
that  a non-compact orientable irreducible $3$-manifold $M$ is a missing 
boundary manifold  if and only if for any compact submanifold $C$ 
of $M$ the fundamental group of any connected component of $M - C$  is finitely
generated. This observation made it possible to reduce the missing
boundary manifold conjecture  to a group-theoretic problem.

Mihalik introduced the notion of a $1$-tame pair of groups in \cite{Mi1}. 
His approach resulted in various group-theoretical results which implied some
special cases of the tame ends conjecture, (cf. \cite{Mi1}, \cite{Mi2}).
The author  introduced in \cite{Gi} a different, but equivalent, definition of a 
tame subgroup. 

Let $H$ be a subgroup of a group $G$.
Choose the  presentation $G= \langle X|R \rangle $. 
Let $T$ be the standard presentation $2$-complex of $G$, i.e. $T$ has one 
vertex, $T$ has an edge, which is a loop, for any generator $x \in X$, and $T$ has a $2$-cell
for any relator $r \in R$. The Cayley complex of $G$, denoted by $Cayley_2(G)$, is  
the universal cover of $T$. Denote by $Cayley_2(G,H)$ the cover of $T$ corresponding to
a subgroup $H$ of $G$.

\begin{definition}
A finitely generated subgroup $H$ of a finitely presented group  
$G$ is tame in $G$ if for any finite subcomplex
$C$ of $Cayley_2(G,H)$ and for any component $K$ of $Cayley_2(G,H) -C$ the group
$\pi_1(K)$ is finitely generated.
\end{definition}

The concept of a tame subgroup is  of independent 
interest. For example, it is not known if there exists a finitely generated 
subgroup $H$ of a finitely presented group $G$ 
such that $H$ is not tame in $G$.

Mihalik demonstrated in \cite{Mi1} that Definition 1 is equivalent to the definition
of a $1$-tame pair $(G,H)$ given in \cite{Mi1}.
He also showed in \cite{Mi1} that $H$ is tame in $G$ if for one large finite subcomplex $C$
of $Cayley_2(G,H)$ the fundamental groups of the connected components of
$Cayley_2(G,H) - C$ are finitely generated. The complex
$C$ can be chosen as a ball around the basepoint $H \cdot 1$.

The following resut was proved by Mihalik in \cite{Mi1}.

\begin{theorem}(Mihalik1)
Let $X$ be a finite polyhedron with $\pi_1(X)=G$ and let $\tilde{X}$ be its universal cover.
Let $H$ be a finitely generated subgroup of $G$ and let $H \setminus \tilde{X}$ be
the  quotient of $\tilde{X}$ by the action of $H$. Then $H$ is tame in $G$ if and only if
for each finite subcomplex $C$ of $H \setminus \tilde{X}$ the fundamental group of every
connected component of $((H \setminus \tilde{X}) - C)$ is finitely generated.
\end{theorem}

Theorem (Mihalik1) implies that  Definition 1 is independent of a finite
presentation of the group $G$.
 
Following standard terminology, we introduce the following definition:

\begin{definition}
A group $G$ is locally tame if all finitely generated subgroups of $G$
are tame in $G$.
\end{definition}

The main result of this paper is the following theorem.

\begin{theorem}
If $H$ and $K$ are finitely presented locally tame groups, then
the free product $G = H*K$ is locally tame.
\end{theorem}
 
\section{Preliminaries}

\begin{remark}
It is not known if the trivial subgroup is tame in any finitely presented group.
Moreover, it is not known if there exists an infinitely presented group 
such that the trivial subgroup is not tame in it.
\end{remark}

\begin{remark}
The trivial subgroup is
tame in any finitely generated abelian group (possibly with torsion) $A$.
Indeed, $A$ is the fundamental group of a finite polyhedron $X$ of the form
$(S^1)^n  \times X_{m_1} \times \cdots \times X_{m_k}$,
where $X_{m_i}$ is homeomorphic to a circle with a $2$-disc attached by a degree
$m_i$ map. The universal cover  $\tilde{X}$ of $X$ is homeomorphic to $R^n \times M$,
where $M$ is a compact set. Let $C$ be a compact subset of $\tilde{X}$. We want to show
that all components of $\tilde{X}-C$ have finitely generated fundamental groups.
We can replace $C$ by a larger compact set $K=p^{-1}p(C)$, where $p$ is the projection map
from $R^n \times M$ onto $R^n$. The space $(R^n \times M) -K$ has components $(R^n-p(C))\times M$,
which have finitely generated fundamental groups. Hence
Theorem (Mihalik1) implies the result.
\end{remark}

\begin{remark}
The trivial subgroup is tame in any negatively curved group, cf. \cite{Gi} and \cite{Mi1}.
\end{remark}

\begin{remark}
If $H$ is a finite index subgroup of a finitely presented group $G$,
then the trivial subgroup is tame in $H$ if and only if it is tame in $G$.

Indeed, let $X$ be a finite polyhedron with $\pi_1(X)=G$ and let $X_H$
be a finite cover of $X$ with $\pi_1(X_H)=H$. Note that the
universal cover $\tilde{X}$ of $X$ is also a universal cover of $X_H$ and $X_H$ is
a finite polyhedron. Hence Theorem (Mihalik1) implies the result.
\end{remark}

\begin{lemma}
Let $G=K \times Z$, where $K$ is a locally tame finitely presented group.
Let $t$ be a generator of $Z$ and let $H$ be a finitely presented subgroup of $G$
of the following form: $H= (H \cap K) \times t^i, i > 0$. 
If $H \cap K$ is finitely generated, then $H$ is tame in $G$.
\end{lemma}

\begin{proof}
Let $Y$ be a finite polyhedron with $\pi_1(Y)=K$ and let $X=Y \times S^1$.
Then $\pi_1(X)=G$.

Let $\tilde{X}$ be the universal covering space of $X$ and let $\tilde{Y}$ be the universal covering space
of $Y$. Note that $\tilde{X}$ is homeomorphic to $\tilde{Y} \times R$.
The quotient of $\tilde{X}$ by the action of $H$ is homeomorphic to $((H \cap K) \setminus \tilde{Y} )\times S^1$.

Let $C_0$ be a finite subcomplex of $((H \cap K) \setminus \tilde{Y} )\times S^1$. 

Enlarge $C_0$ to a finite complex $C$ of the
form $C=C_1 \times S^1$, where $C_1$ is a finite subcomplex of $(H \cap K) \setminus \tilde{Y} $. Note that
$((H \cap K) \setminus \tilde{Y} )\times S^1 - C$ is homeomorphic to $(((H \cap K) \setminus \tilde{Y} ) -C_1) \times S^1$. 
By assumptions $H \cap K$ is finitely generated and $K$ is locally tame, hence the fundamental group of each component of
$((H \cap K) \setminus \tilde{Y} )-C_1$ is finitely generated,
therefore the fundamental group of each component of

$(((H \cap K) \setminus \tilde{Y}) -C_1)\times S^1$ is finitely generated.
Hence Theorem (Mihalik1) implies that $H$ is tame in $G$.
\end{proof}

\begin{remark}
Note that an infinitely generated subgroup might not be tame in a finitely presented group.
For example, let $F= \langle a,b \rangle$ be a free group of rank $2$ and let
$H=\langle a^n b a^{-n}, n \in Z \rangle$ be a subgroup of $F$. As the fundamental group of the 
complement of any finite subcomplex of $Cayley_2(F,H)$ is infinitely generated, it follows 
that $H$ is not tame in $F$.
\end{remark}

\begin{remark}
Note that in general quasiisometry does not preserve tameness.

Let $A$ be a free abelian group of rank two.
As was noted in Remark 2, the trivial subgroup is tame
in $A$. The Cayley complex of $A$ is homeomorphic to the Euclidean plane.
The Euclidean plane is quasiisometric to its subset $S$ consisting of the 
horizontal lines $\{y=n, n \in Z \}$ and the vertical lines 
$\{ x=n, n \in Z \}$. The set $S$ is homeomorphic to the Cayley complex $Cayley_2(F,F')$,
where $F$ is the free group of rank $2$ and $F'$ is its commutator subgroup.

Hence $Cayley_2(A)=Cayley_2(A,1)$ is quasiisometric to $Cayley_2(F, F')$. However,
as the complement of any compact subset of $S$ in $Cayley_2(F, F')$ has infinitely generated
fundamental group, $F'$ is not tame in $F$.
\end{remark}

\begin{remark}
Let $G$ and $G_0$ be finitely presented groups,
let $H$ be a finitely generated subgroup of $G$, and let $H_0$ 
be a finitely generated subgroup
of $G_0$. 

Assume that $Cayley_2(G,H)$ is quasiisometric to
$Cayley_2(G_0,H_0)$.

It is not known if  $H$ is tame in $G$ if and only if
$H_0$ is tame in $G_0$.
\end{remark}

\begin{remark}
Let $G$ and $G_0$ be finitely presented quasiisometric groups.
Let $H$ be a finitely generated subgroup of a group $G$ and let 
$H_0$ be a finitely generated subgroup of $G_0$. 

Assume that $H$ is quasiisometric to $H_0$. It is not known if $H$ is tame in $G$
if and only if $H_0$ is tame in  $G_0$.

This question is open even in the special case when $G = G_0$.
\end{remark}

\begin{remark}
A finite group is locally tame, because its
Cayley complex is finite.
\end{remark}

\begin{remark}  
A finite index  subgroup $H$ of a finitely presented
group $G$ is tame in $G$ because the complex $Cayley_2(G,H)$ 
is finite.
\end{remark}

\begin{remark}
The trivial subgroup is tame in $Z$ because $Cayley_2(Z)$ is homeomorphic to
a straight line.
As any non-trivial subgroup of $Z$ has finite index in $Z$,
it follows that $Z$ is locally tame.
\end{remark}

\begin{remark}
Free groups are locally tame. Indeed, for any free group $F$ and its
finitely generated subgroup $H$ the complex $Cayley_2(F,H)$
is one-dimensional. When $H$ 
is finitely generated, $Cayley_2(F,H)$ is homotopic to
a wedge of finitely many circles.
\end{remark}
  
The author proved the following result in \cite{Gi}.

\begin{theorem} (Gitik)
Let $N$ be a finitely generated normal subgroup of a 
finitely generated group $G$.
Then $N$ is tame in $G$  if the trivial subgroup 
is tame in the factor group $G/N$.
\end{theorem}

Theorem (Gitik) implies the following fact.

\begin{lemma}
A finitely generated abelian group is locally tame.
\end{lemma}
\begin{proof}
Any finitely generated abelian group $A$ is a direct product
of the form $A = Z^n \times Tor$, where $Z^n$ is the direct product
of $n$ copies of $Z$ and $Tor$ is a finite abelian group.
If $A$ is finite, it is locally tame by Remark 9.
If $A$ is infinite and $N$ is a subgroup of $A$, the factor group $A/N$ is a finitely
generated abelian group. Remark 2 states that the trivial
subgroup is tame in $A/N$, hence it follows from Theorem (Gitik) that 
$N$ is tame in $A$.
\end{proof}

The following property is closely related to tameness.

\begin{definition}
A $CW$-complex $W$ has property $*$ if for any finite subcomplex $C$ the fundamental group
of each component of $W-C$ is finitely generated.
\end{definition}

\begin{lemma}
A $CW$-complex $W$ has property $*$ if and only if any finite cover of $W$ has property $*$.
\end{lemma}
\begin{proof}
Assume $W$ has property $*$. Let $U$ be a finite cover of $W$ and let $C$ be a finite subcomplex of $U$.
Let $p:U \rightarrow W$ be the covering map. Then $p(C)$ is a finite 
subcomplex of $W$ and $C \subset p^{-1}(p(C))$. Each component of $U-p^{-1}(p(C))$ is a finite cover of 
a component of $W-C$, hence the fundamental group of each component of $U-p^{-1}(p(C))$ is a finite index subgroup
of a component of $W-C$. As $W$ has property $*$, it follows that the fundamental group of 
each component of $U-p^{-1}(p(C))$ is finitely generated. However $U-C$ is obtained from $U-p^{-1}(p(C))$
by adding a finite complex, hence $U$ has property $*$.

Assume that $U$ has property $*$ and let $K$ be a finite subcomplex of $W$. Then $p^{-1}(K)$ is 
a finite subcomplex of $U$ and each component of $U-p^{-1}(K)$ has a finitely generated fundamental group.
Each component of $W-U$ is finitely covered by a component of $U-p^{-1}(K)$, hence the fundamental group
of each component of $U -p^{-1}(K)$ is a finite index subgroup of a component of $W-K$, hence $W$
has property $*$.
\end{proof}

Mihalik proved the following result in \cite{Mi2}.

\begin{theorem} (Mihalik2)
Let $1 \rightarrow A \rightarrow G \rightarrow B \rightarrow 1$
be a short exact sequence of infinite finitely presented groups,
and let $H$ be a finitely generated subgroup of $A$ of infinite 
index in $A$. Then $H$ is tame in $G$.
\end{theorem}

\begin{lemma}
Let $G_0$ be a finite index subgroup of a finitely presented group $G$.
Then a finitely generated subgroup $H$ of $G$ is tame in $G$ if and only if 
$H_0=H \cap G_0$ is tame in $G_0$.
\end{lemma}
\begin{proof}
Let $X$ be a $CW$-complex with $\pi_1(X) = G$ and let $X_0$ be a finite cover of $X$
with $\pi_1(X_0)= G_0$. Let $X_H$ be a cover of $X$ with $\pi_1(X)=H$ and let $X_{H_0}$
be a finite cover of $X_H$ with $\pi_1(X_{H_0})=H_0$.

Consider the following commutative diagram of covering spaces, where the maps are covering projections.

\[
\begin{CD}
X_{H_0}  @>>> X_0  \\
@VVV          @VVV\\
X_H      @>>> X
\end{CD}
\]

Assume that $H$ is tame in $G$. Theorem (Mihalik1) implies that $X_H$ has property $*$.
Then Lemma 4 implies that $X_{H_0}$ has property $*$ and Theorem (Mihalik1) implies that
$H_0$ is tame in $G_0$.

Assume that $H_0$ is tame in $G_0$. Theorem (Mihalik1) implies that $X_{H_0}$ has property $*$.
Then Lemma 4 implies that $X_H$ has property $*$, and Theorem (Mihalik1) implies that $H$ is tame in $G$.
\end{proof}

The author proved the following result in \cite{Gi}.

\begin{lemma} (Gitik)
Let $H$ be a subgroup of a group $G$, and let $H_0$
be a finite index subgroup of $H$. Then $H$ is tame in $G$ if and only if
$H_0$ is.
\end{lemma}

Theorem (Mihalik2), Theorem (Gitik), and Lemma (Gitik) imply the following results.

\begin{remark}
Let $N$ be a normal finitely generated subgroup of a finitely presented
group $G$.
\begin{enumerate}

\item  
If $G/N$ is infinite, then any finitely generated subgroup
of $N$ is tame in   $G$, provided 
the trivial subgroup is tame in $G/N$.
 
Indeed, Lemma (Gitik)  and Theorem (Gitik) imply the result for finite-index
subgroups of $N$, and Theorem (Mihalik2) implies  the result for 
infinite-index subgroups of $N$, provided  $G/N$ is infinite.

\item 
Lemma (Gitik) implies that if $G/N$ is finite, then any finitely generated subgroup of $N$
is tame in $G$ if and only if it is tame in $N$.
\end{enumerate}
\end{remark}

\begin{lemma}
Let $H$ and $K$ be infinite finitely presented groups, and let $G = H \times K$.
Let $S$ be a finitely generated subgroup of $G$. If the trivial subgroup is tame in $H$
and in $K$, then $S \cap H$ and $S \cap K$ are tame in $G$. 
\end{lemma}

\begin{proof}
Consider the short exact sequence
$1 \rightarrow H \rightarrow G \rightarrow K \rightarrow 1$.
If $S \cap H$ has infinite index in $H$, the result follows from Theorem (Mihalik2).
If $S \cap H$ has finite index in $H$, Lemma (Gitik) implies that $S \cap H$ is tame in $G$
if and only if $H$ is tame in $G$. Theorem (Gitik) implies that $H$ is tame in $G$ if the
trivial subgroup is tame in $G/H = K$.
\end{proof}

\begin{lemma}
Let $H$ and $K$ be finitely generated groups and let $G = H \times K$.
If the trivial subgroup is tame in $H$ and in $K$ then it is tame in $G$.
\end{lemma}
\begin{proof}
Indeed, if $H$ and $K$ are infinite the result follows from Theorem (Mihalik2).
If both $H$ and $K$ are finite, the result follows from Remark 9. 
If $H$ is finite and $K$ is infinite the result follows from Remark 4.
\end{proof}

\begin{remark}
Let $G=H \times Z$. As the trivial subgroup is tame in $Z$,
Theorem (Gitik) implies that $H$ is tame in $G$.

Let $H_0$ be a finitely presented subgroup of $H$.
If $H_0$ is of infinite index in $H$, Theorem (Mihalik2) states that $H_0$ is tame in $G$.
If $H_0$ is of finite index in $H$, Lemma (Gitik) states that $H_0$ is tame in $G$.
\end{remark}

\section{Proof of Theorem 2}

We need the following notation.

Let $X^* = \{x,x^{-1} |x \in X \}$. For  $x \in X$ define $(x^{-1})^{-1} =x$.
 
Let $G$ be a group generated by a set $X$ and let $H$ be a subgroup of $G$.
Let $\{Hg \}$ be the set of right cosets  of $H$ in $G$.
 
The coset graph of $G$  with respect to $H$, denoted $Cayley(G,H)$,  is the oriented graph 
whose vertices are the cosets $\{Hg \}$, the set of edges is $ \{Hg \} \times X^*$, 
and an edge $ (Hg, x)$ begins at the vertex  $Hg$ and ends at the 
vertex $Hgx$. Denote the Cayley graph of $G$ by $Cayley(G)$. 
Note that $Cayley(G,H)$ is the quotient of $Cayley(G)$ by left 
multiplication by $H$. Also note that  the $1$-skeleton of $Cayley_2(G)$ is
$Cayley(G)$, and the $1$-skeleton of $Cayley_2(G,H)$ is $Cayley(G,H)$.

Let $G$ be generated by a disjoint union of sets $X$ and $Y$.
We call a subset $C$ of $Cayley(G,H)$ an $X$-component, if all edges of $C$
have the form $(Hg, x)$ with $x \in X^*$. We call $C$ a $Y$-component if all edges of $C$
have the form $(Hg, y)$ with $y \in Y^*$.

\begin{proof}
Let $S$ be a finitely generated subgroup of $G=H*K$, where $H$ is generated by the set $X$ and 
$K$ is generated by the set $Y$. We assume that $X$ and $Y$ are disjoint.

The Kuros' subgroup theorem states that $S$ is a free product of a free group with subgroups
of conjugates of $H$ or $K$. It follows that $S$ is the fundamental group of a graph of groups, where the vertex groups are subgroups 
of conjugates of $H$ or $K$ and the edge groups are trivial, cf.\cite{S-W}. 

Note that the complex $Cayley_2(G,S)$ has the following structure.
All maximal $X$-components of $Cayley_2(G,S)$, except for a finite number, are homeomorphic to 
$Cayley_2(H)$. The remaining finitely many maximal $X$-components of $Cayley_2(G,S)$ are homeomorphic 
to $Cayley_2(H, H_i)$, where $H_i, 1 \le i \le n$ is a finitely generated subgroup of $H$. All maximal $Y$-components of
$Cayley_2(G,S)$, except for a finite number, are homeomorphic to 
$Cayley_2(K)$. The remaining finitely many maximal $Y$-components of $Cayley_2(G,S)$ are homeomorphic 
to $Cayley_2(K, K_i)$, where $K_i, 1 \le i \le n$ 
is a finitely generated subgroup of $K$.

Let $C$ be a compact subset of $Cayley_2(G,S)$. Then $C$ has non-empty intersection
with only finitely many maximal $X$-components and maximal $Y$-components of $Cayley_2(G,S)$.
As $H$ and $K$ are locally tame, the fundamental groups of each component of the complement of $C$ in any maximal $X$-component 
and any maximal $Y$-component is finitely generated. Hence the fundamental group of $Cayley_2(G,S) - C$ is finitely generated.
\end{proof}

\section{Open Questions}

\begin{enumerate}
\item

Let $G$ be a finitely presented group which is a direct product
of subgroups $H$  and $K$, and let $S$ be a 
finitely presented subgroup of $G$ such that $S \cap H$ and $S \cap K$ 
are finitely generated. Is it true that $S$ is tame in $G$ if and only if $S \cap H$ is tame in $H$ and
$S \cap K$ is tame in $K$? 
  
\item  
If $H$ is locally tame, does it follow that a cyclic extension of $H$ is locally tame?

A special case of this question: is $G=H \times Z$ locally tame?

Are finitely generated polycyclic groups,  in particular, finitely generated 
nilpotent groups, locally tame?

\item
Is the trivial subgroup tame in the Thompson group $F$ (cf. \cite{C-F}) with the infinite
presentation 
$ \langle x_0,x_1, x_2, \cdots |x_k^{-1}x_nx_k=x_{n+1}, k <n \rangle$?

Is the trivial subgroup tame in the Thompson group $F$ with the finite presentation
$\langle A, B|[AB^{-1},A^{-1}BA], [AB^{-1}, A^{-2}BA^2 \rangle $?
   
\item  Let $N$ be a finitely generated normal subgroup of a finitely presented group $G$.
Assume that $N$ and $G/N$ are locally tame and $N$ has fgip in $G$.
Is $G$ locally tame?

\end{enumerate} 

\begin{remark}
Recall that a finitely generated subgroup $K$ has fgip in $G$  if
the intersection of $K$ with any finitely generated subgroup of $G$
is finitely generated.

The following result would follow from an affirmative answer to Question 1.

Let $G,H$ and  $K$  be as in Question 1. Assume that  $H$ 
and $K$ are locally tame. If  $H$ and $K$ have fgip in $G$,
then $G$ is locally tame.
  
Indeed, let $S$ be a finitely generated subgroup of $G$. As $H$ and $K$ have fgip in $G$,
the intersections $S \cap H$ and $S \cap K$ are finitely generated. As $H$ and $K$
are locally tame, it follows that $S \cap H$ is tame in $H$ and $S \cap K$ is tame in $K$.
Hence the affirmative answer to Question 1 would imply that $S$ is tame in $G$. As $S$ is an arbitrary finitely
generated subgroup of $G$, it follows that $G$ is locally tame.
\end{remark}
 
\section{Acknowledgment}
The author wants to thank Peter Scott for helpful conversations.

\end{document}